\documentclass[11pt]{article}
\usepackage{cite}
\usepackage{amsfonts}
\usepackage{mathrsfs}
\usepackage{amssymb,amsmath,mathrsfs, amsthm}
\usepackage{amssymb, palatino}
\usepackage{graphicx}
\usepackage{mathtools}

\numberwithin{equation}{section}

\newtheorem{theorem}{Theorem}[section]

\newtheorem{corollary}[theorem]{Corollary}
\newtheorem{proposition}[theorem]{Proposition}
\theoremstyle{definition}
\newtheorem{definition}[theorem]{Definition}

\theoremstyle{remark}

\date{}
\begin{document}

\title{On strongly convex projectively flat and dually flat complex Finsler metrics}
\author{Hongchuan Xia$^1$\quad\quad Chunping Zhong$^2$
\\
{\small $^1$ College of Mathematics and Statistics, Xinyang Normal University, Xinyang  464000, China}\\
{\small $^2$ School of Mathematical Sciences, Xiamen
University, Xiamen 361005, China}\\
 }

\date{}

\maketitle
\begin{abstract}
In this paper, we prove that a strongly convex complex Finsler metric $F$  on a domain $D\subset\mathbb{C}^n$ is projectively flat (resp. dually flat) if and only if $F$ comes from a strongly convex complex Minkowski metric.
\end{abstract}

\section{Introduction and main theorem}
In 1900, Hilbert announced his famous 23 problems \cite{Hilbert}. The Hilbert's Fourth Problem is to characterize the distance functions on an open subset in $\mathbb{R}^n$ such that straight lines are shortest paths. Distance functions induced by Finsler metrics are regarded as smooth cases. Therefore the
Hilbert's Fourth Problem in smooth case is to  characterize Finsler metrics on an open subset in $\mathbb{R}^n$ whose geodesics are straight lines. Finsler metrics with this property are called projectively flat Finsler metrics \cite{CS}. In 1903, Hamel \cite{H} proved that a real Finsler metric $F(x,u)$ defined on an open subset $D\subset\mathbb{R}^n$ is projectively flat if and only if
\begin{equation}
\sum_{b=1}^n\frac{\partial^2F}{\partial x^b\partial u^a}u^b=\frac{\partial F}{\partial x^a},\quad \mbox{for}\;a=1,\cdots,n,\label{pf}
\end{equation}
where $x=(x^1,\cdots,x^n)\in D$ and $u=(u^1,\cdots,u^n)\in T_xD$.
It is still difficult to understand the local metric structure of such metrics, and many geometers are interested in this kind of metrics, especially projectively flat real Finsler metrics with constant flag curvatures \cite{ChengShen, Ber1, Ber2, Bryant1, Bryant2, Shen, Mo, Zhou, LiB}. Recently, Li \cite{LiB} obtained a classification of locally projectively flat real Finsler metrics with constant flag curvatures.

The notion of dually flat real Finsler metrics was first introduced by  Amari and Nagaoka \cite{Am} in information geometry. Several years later, Shen \cite{Sh} extended this notion to real Finsler geometry and proved that a real Finsler metric $F(x,u)$ defined on an open subset $D\subset\mathbb{R}^n$ is dually flat if and only if $F$ satisfies
\begin{equation}
\sum_{b=1}^n\frac{\partial^2F^2}{\partial x^b\partial u^a}u^b=2\frac{\partial F^2}{\partial x^a},\quad\mbox{for}\;a=1,\cdots,n.\label{df}
\end{equation}
The study of dually flat real Finsler metrics also attracts a lot of interests  \cite{Cheng, Xia, Tayebi, LiB2, Yu, Yu2, Huang}.

 By \eqref{pf} and \eqref{df}, it  follows that every real Minkowski metric is a projectively flat (resp. dually flat) metric. There are, however, lots of real Finsler metrics which are not come from real Minkowski metrics, but they are projectively flat (resp. dually flat).

 Note that complex Finsler metrics such as the Kobayashi and Carath$\acute{\mbox{e}}$odory metrics arise naturally in geometric function theory in several complex variables \cite{AP}. In Hermitian geometry, a Hermitian metric on a complex manifold is necessarily a Riemannian metric which is compatible with the complex structure of the ambient complex manifold. In complex Finsler geometry, complex Finsler metrics are not tightly related to real Finsler metrics. A complex Finsler metric is not necessarily a real Finsler metric, a real Finsler metric on a complex manifold is not necessarily a complex Finsler metric. Recently in \cite{Zhong3,XZ3,XZ4}, however, the authors showed that there are lots of strongly convex complex Finsler metrics, i.e., if one neglecting the complex structure and rewriting them in terms of real coordinates on the tangent bundle, these metrics  are also real Finsler metrics on complex manifolds.

 One may wonder whether there are strongly convex complex Finsler metrics which are projectively flat (resp. dually flat). Compared to real Finsler geometry,  we obtain the following main theorem, which is extremely simple.

\begin{theorem}
Let $F$ be a strongly convex complex Finsler metric on an open domain $D\subset\mathbb{C}^n$. Then

(i) $F$ is projectively flat if and only if it comes from a strongly convex complex Minkowski metric;

(ii) $F$ is dually flat if and only if it comes from a strongly convex complex Minkowski metric.
\end{theorem}

\section{Preliminaries }

\begin{definition}\cite{AP}
A real Minkowski norm on $\mathbb{R}^n$ is a  function $F: \mathbb{R}^n\rightarrow [0,+\infty)$ satisfying

 (C1)\quad $F$ is $C^\infty$ on $\mathbb{R}^n\setminus \{0\}$;

 (C2)\quad $F(\lambda u)=|\lambda|F(u)$ for all $\lambda\in\mathbb{R}$ and $u=(u^1,\cdots,u^n)\in \mathbb{R}^n$;

 (C3)\quad $F$ is strongly convex, i.e., the matrix
$$g_{ij}(u):=\frac{1}{2}\frac{\partial^2F^2}{\partial u^i\partial u^j}(u)$$
is positive definite on $\mathbb{R}^n\setminus\{0\}$.
\end{definition}

\begin{definition}\cite{AP}
A real Finsler metric $F$ on a smooth  manifold $M$ is a  function $F: TM\rightarrow[0,+\infty)$ satisfying

 (C1)\quad $F$ is $C^\infty$ on $TM\setminus\{\mbox{zero section}\}$;

 (C2)\quad  The restriction of $F$ to $T_pM\cong\{p\}\times\mathbb{R}^n$ is a real Minkowski norm  for each $p\in M$.
\end{definition}

\begin{definition}\label{def1}\cite{AP}
A complex Minkowski norm on $\mathbb{C}^n$ is a  function $F: \mathbb{C}^n\rightarrow [0,+\infty)$ satisfying

(C1)\quad $F$ is $C^\infty$ on $\mathbb{C}^n\setminus\{0\}$;

(C2)\quad $F(\lambda v)=|\lambda|F(v)$ for all $\lambda\in\mathbb{C}$ and $v=(v^1,\cdots,v^n)\in \mathbb{C}^n$;

(C3)\quad $F$ is strongly pseudoconvex, i.e., the  matrix
$$G_{\alpha\overline{\beta}}(v):=\frac{\partial^2F^2}{\partial v^\alpha\partial\overline{v^\beta}}(v)$$
is positive definite on $\mathbb{C}^n\setminus\{0\}$.
\end{definition}

\begin{definition}\cite{AP}
A strongly pseudoconvex complex Finsler metric $F$ on a complex  manifold $M$ is a  function $F: T^{1,0}M\rightarrow [0,+\infty)$ satisfying

(C1)\quad $F$ is $C^\infty$ on $T^{1,0}M\setminus\{\mbox{zero section}\}$;

(C2)\quad The restriction of $F$ to $T_p^{1,0}M\cong\{p\}\times \mathbb{C}^n$ is a complex Minkowski norm  for each $p\in M$.
\end{definition}

Let $M$ be a complex $n$-dimensional manifold with the canonical complex structure $J$. Denote by $T_{\mathbb{R}}M$ the real tangent
bundle, and $T_{\mathbb{C}}M$ the complexified tangent bundle of $M$. Then $J$ acts complex linearly on $T_{\mathbb{C}}M$ so that $T_{\mathbb{C}}M = T^{1,0}M\oplus T^{0,1}M$,
 where $T^{1,0}M$ is called the holomorphic tangent bundle of $M$,
and we also denote by $J$ the naturally induced complex structure on $T^{1,0}M$ if it causes no confusion.
Let $\{z^1,\cdots, z^n\}$ be a set of local complex coordinates on $M$, with $z^i = x^i+\sqrt{-1}x^{i+n}$, so that $\{x^1,\cdots,x^n,x^{1+n}, \cdots, x^{2n}\}$
are local real coordinates on $M$. Denote by $\{z^1,\cdots,z^n,v^1,\cdots,v^n\}$ the induced complex coordinates on $T^{1,0}M$, with $v^i=u^i +\sqrt{-1}u^{i+n}$,
so that $\{x^1,\cdots, x^{2n}, u^1,\cdots, u^{2n}\}$ are local real coordinates on $T_{\mathbb{R}}M$.

The bundles $T^{1,0}M$ and $T_{\mathbb{R}}M$ are isomorphic. We choose the explicit isomorphism $^o: T^{1,0}M\rightarrow T_{\mathbb{R}}M$ with
its inverse $_o: T_{\mathbb{R}}M\rightarrow T^{1,0}M$, which are respectively given by
\begin{equation*}
T_{\mathbb{R}}M\ni u=v^o=v+\overline{v},\;\;\;\;\forall v\in T^{1,0}M
\end{equation*}
and
\begin{equation*}
T^{1,0}M\ni v=u_o=\frac{1}{2}(u-\sqrt{-1}Ju),\;\;\;\;\forall u\in T_{\mathbb{R}}M.
\end{equation*}

Let $F: T^{1,0}M\rightarrow [0,+\infty)$ be a complex Finsler metric on a complex manifold $M$. Using the complex structure $J$ on $M$ and
the bundle map $_o: T_{\mathbb{R}}M\rightarrow T^{1,0}M$ we can define a real function
\begin{equation}
F^o: T_{\mathbb{R}}M\rightarrow [0,+\infty),\;\;F^o(u):=F(u_o),\;\;\forall u\in T_{\mathbb{R}}M.
\end{equation}
\begin{definition}\cite{AP}
A complex Finsler metric $F$ is called strongly convex if the associated function $F^o$ is a real Finsler metric.
\end{definition}
Since the the strong convexity of the $F^o$-indicatrix
 $I_{F^o}(p)=\{u\in T_{\mathbb{R},p}M|F^o(u)<1\}$ implies the strong pseudoconvexity of  $F$-indicatrix $I_{F}(p)=\{v\in T^{1,0}_{p}M|F(v)<1\}$, a strongly convex complex Finsler metric $F$ is necessarily strongly pseudoconvex. The converse, however, is not true.
 In the following, we use the same symbol $F$ to denote the associated real Finsler metric $F^o$ with
 the understanding that $F(u)$ is defined by $F(u_o)$ for $u\in T_{\mathbb{R}}M$.

 Obviously, every real Minkowski norm is a projectively flat and dually flat metric. A famous projectively flat and dually flat Finsler metric is the Funk metric $F = F(x, u)$ on a strongly convex domain $D\subset\mathbb{R}^n$ which is neither a Riemannian metric  nor a real Minkowski norm and of constant flag curvature $K = -\frac{1}{4}$. Especially, when $D$ is the unit ball $\mathbb{B}^n\subset \mathbb{R}^n$, the Funk metric is given by
\begin{equation}
F(x, u)=\frac{\sqrt{(1-|x|^2)|u|^2+\langle x \,|\, u \rangle^2}}{1-|x|^2}+\frac{\langle x \,|\, u\rangle}{1-|x|^2},\label{funk}
\end{equation}
where $|\cdot|$ and $\langle\cdot \,|\, \cdot \rangle $ denote respectively the real Euclidean norm and real Euclidean inner product.

In general, a real Finsler metric $F$ on an open domain $D \subset \mathbb{R}^{2n}\cong \mathbb{C}^n$ is not necessarily a complex Finsler metric. In fact, consider the Funk metric $F=F(x, u)$ of the form \eqref{funk} on the unit ball $\mathbb{B}^{2n}\subset\mathbb{C}^n$, we can rewritten it with respect to the complex coordinates as follows
\begin{equation}
F(z,v)=\frac{\sqrt{(1-\|z\|^2)\|v\|^2+(\mbox{Re}\langle z,v \rangle)^2}}{1-\|z\|^2}+\frac{\mbox{Re}\langle z,v\rangle}{1-\|z\|^2},\label{funk1}
\end{equation}
where $\mbox{Re}\langle z , v\rangle$ denotes the real part of $\langle z , v\rangle$, $\| \cdot \|$ and $\langle\cdot, \cdot \rangle $ denote respectively the complex Euclidean norm and the complex Euclidean inner product. It is easy to check that the function \eqref{funk1} does not satisfy the homogeneous condition (C2) in Definition \ref{def1}.

Recently, by a systematical investigation on unitary invariant metrics \cite{Zhong3,XZ1,XZ2}, general complex $(\alpha, \beta)$ metrics \cite{XZ3} and strongly convex complex Finsler metrics \cite{XZ4}, the authors constructed lots of strongly convex complex Finsler metrics. Thus for a given $n$-dimensional complex manifold $M$ endowed with a
strongly convex complex Finsler metric $F$, if neglecting the complex structure on $M$, we can view it as a $2n$-dimensional real Finsler manifold. A natural question is, whether there exists a strongly convex complex Finsler metric which is also projectively flat or dually flat in the real Finsler sense? If exists, whether there exists a non complex Minkowski metric among those strongly convex complex Finsler metrics? The purpose of this paper is to answer these two questions.

\section{Proof of the main theorem}
In \cite{Zhong3}, the second author obtained a complex version of characterization for strongly convex complex Finsler metrics to be projectively flat and
 dually flat, which corresponds respectively to the PDEs systems \eqref{pf} and \eqref{df} obtained in real case. In the following we shall use the Einstein summation convention, i.e., the same indexes up and down are summed over their
ranges. In order to prove our result, we need the following proposition.

\begin{proposition}\label{pr1}\cite{Zhong3}
Let $F(z,v)$ be a strongly convex complex Finsler metric on an open domain $D\subset\mathbb{C}^n$. Then

(1) $F$ is projectively flat if and only if
\begin{equation}
\frac{\partial F}{\partial z^i}=\frac{\partial^2F}{\partial z^j\partial v^i}v^j+\frac{\partial^2 F}{\partial \overline{z^j}\partial v^i}\overline{v^j},\quad\quad\forall\; i=1,\cdots,n.\label{pfc}
\end{equation}

(2) $F$ is dually flat if and only if
\begin{equation}
2\frac{\partial F^2}{\partial z^i}=\frac{\partial^2F^2}{\partial z^j\partial v^i}v^j+\frac{\partial^2F^2}{\partial \overline{z^j}\partial v^i}\overline{v^j},\quad\quad\forall \;i=1,\cdots,n.\label{dpfc}
\end{equation}
\end{proposition}

\begin{theorem}
Let $F$ be a strongly convex complex Finsler metric on an open domain $D\subset\mathbb{C}^n$. Then

(i) $F$ is projectively flat if and only if it comes from a strongly convex complex Minkowski metric;

(ii) $F$ is dually flat if and only if it comes from a strongly convex complex Minkowski metric.
\end{theorem}
\begin{proof} If $F(z,v)$ comes from a complex Minkowski metric, then $F$ is independent of the variable $z=(z^1,\cdots,z^n)$. Thus $F$ satisfies \eqref{pfc} and \eqref{dpfc}, that is, $F$ is  projectively flat and dually flat.

Conversely, if $F$ is projectively flat, then \eqref{pfc} holds. Contracting with $v^i$ in both sides of \eqref{pfc} yields
\begin{equation}
\frac{\partial F}{\partial z^i}v^i=\frac{1}{2}\frac{\partial F}{\partial z^j}v^j+\frac{1}{2}\frac{\partial F}{\partial\overline{z^j}}\overline{v^j},\label{e2}
\end{equation}where we used the homogeneity of $F$, i.e. $F^2(z,\lambda v)=\lambda\overline{\lambda}F^2(z,v)$ for $\lambda\in\mathbb{C}-\{0\}$. By \eqref{e2}, we get
\begin{equation}
\frac{\partial F}{\partial z^i}v^i=\frac{\partial F}{\partial \overline{z^i}}\overline{v^i}.\label{e}
\end{equation}
Differentiating \eqref{e} with respect to $v^j$ and exchanging indexes $i$ and $j$ gives
\begin{equation}
\frac{\partial F}{\partial z^i}+ \frac{\partial^2F}{\partial z^j\partial v^i}v^j=\frac{\partial^2F}{\partial\overline{z^j}\partial v^i}\overline{v^j}. \label{d}
\end{equation}
By \eqref{pfc} and \eqref{d}, we obtain
\begin{equation}
\frac{\partial^2F}{\partial z^j\partial v^i}v^j=0. \label{c}
\end{equation}
Again contracting with $v^i$ in both sides of \eqref{c} implies
\begin{equation}
\frac{\partial F}{\partial z^j}v^j=0. \label{r}
\end{equation}
Differentiating \eqref{r} with respect to $\overline{v^i}$, we get
\begin{equation}
\frac{\partial^2 F}{\partial z^j\partial\overline{v^i}}v^j=0. \label{ff}
\end{equation}
Taking conjugation in \eqref{ff} shows
\begin{equation}
\frac{\partial^2 F}{\partial \overline{z^j}\partial v^i}\overline{v^j}=0. \label{v}
\end{equation}
Plunge \eqref{c} and \eqref{v} into \eqref{d}, we get $\frac{\partial F}{\partial z^i}=0$ for $i=1, \cdots, n$. By conjugation we get $\frac{\partial F}{\partial\overline{z^i}}=0$ for $i=1,\cdots,n$.  Thus $F(z,v)$ is independent of the base manifold coordinates $z=(z^1,\cdots,z^n)$, i.e.,
$F$ comes from a complex Minkowski metric.

If $F$ is dually flat, then $F$ satisfies \eqref{dpfc}.
 Contracting \eqref{dpfc} with $v^i$ gives
$$2\frac{\partial F^2}{\partial z^i}v^i=\frac{\partial F^2}{\partial z^j}v^j+\frac{\partial F^2}{\partial \overline{z^j}}\overline{v^j},$$
that is
\begin{equation}
\frac{\partial F^2}{\partial z^i}v^i=\frac{\partial^2 F^2}{\partial \overline{z^i}}\overline{v^i}.\label{e1}
\end{equation}
Differentiating \eqref{e1} with respect to $v^j$ and exchanging indexes $i$ and $j$ yields
\begin{equation}
\frac{\partial F^2}{\partial z^i}+ \frac{\partial^2F^2}{\partial z^j\partial v^i}v^j=\frac{\partial^2F^2}{\partial\overline{z^j}\partial v^i}\overline{v^j}. \label{d1}
\end{equation}
By \eqref{dpfc} and \eqref{d1}, we obtain
\begin{equation}
\frac{\partial^2F^2}{\partial \overline{z^j}\partial v^i}\overline{v^j}=3\frac{\partial^2F^2}{\partial z^j\partial v^i}v^j. \label{c1}
\end{equation}
Again contracting with $v^i$ in both sides of \eqref{c1} shows
\begin{equation}
\frac{\partial F^2}{\partial \overline{z^j}}\overline{v^j}=3\frac{\partial F^2}{\partial z^j}v^j, \label{r1}
\end{equation}
which implies
\begin{equation}
\frac{\partial F^2}{\partial z^j}v^j=3\frac{\partial F^2}{\partial \overline{z^j}}\overline{v^j}. \label{r2}
\end{equation}
By \eqref{r1} and \eqref{r2} we obtain
\begin{equation}
\frac{\partial F^2}{\partial z^j}v^j=\frac{\partial F^2}{\partial\overline{z^j}}\overline{v^j}=0. \label{f1}
\end{equation}
So
\begin{equation}
\frac{\partial^2F^2}{\partial \overline{z^j}\partial v^i}\overline{v^j}=0. \label{v1}
\end{equation}
Substituting \eqref{c1} and \eqref{v1} into \eqref{dpfc} gives $\frac{\partial F^2}{\partial z^i}=0$ for $i=1, \cdots, n$. By conjugation, we get $\frac{\partial F^2}{\partial \overline{z^i}}=0$ for $i=1,\cdots,n$. Thus $F^2$ is independent of the base manifold coordinates $z=(z^1,\cdots,z^n)$, i.e.,
$F$ comes from a complex Minkowski metric. This completes the proof.
\end{proof}
\begin{corollary}
Let $F$ be a strongly convex complex Finsler metric on an open domain $D\subset\mathbb{C}^n$. Then
 $F$ is projectively flat if and only it is dually flat.
\end{corollary}

\begin{corollary} Let $F$ be a Hermitian metric defined on a domain $D\subset\mathbb{C}^n$. Then $F$ is projectively flat (resp. dually flat) if and only if $F$ is a  positive scalar multiple of the complex Euclidean inner product.
\end{corollary}

{\bf Acknowledgements:}\; \small This work was supported by the National Natural Science Foundation of China (Grant Nos. 11671330, 11701494, 11571288) and the Nanhu Scholars Program for Young Scholars of Xinyang Normal University.

\end{document}